\documentclass[12pt]{amsart}
\usepackage[colorlinks=true,citecolor=blue,linkcolor=blue, backref=page]{hyperref}
\usepackage{amssymb,verbatim,amscd,amsmath,graphicx,nccmath}
\usepackage{mathrsfs}
\usepackage{upgreek}
\usepackage{enumerate}
\usepackage{graphicx}
\usepackage{caption}
\usepackage{subcaption}
\usepackage{pdfsync}
\usepackage{color,colortbl}
\usepackage{fullpage}
\usepackage{bbm}
\usepackage{comment}
\usepackage{multirow}
\usepackage{enumitem}
\usepackage{comment}
\usepackage{cleveref}
\numberwithin{equation}{section}
\newtheorem{introthm*}{Main Result}
\newtheorem{theorem}{Theorem}[section]
\newtheorem{lemma}[theorem]{Lemma}
\newtheorem{proposition}[theorem]{Proposition}
\newtheorem{corollary}[theorem]{Corollary}

\theoremstyle{definition}
\newtheorem{definition}[theorem]{Definition}

\newtheorem{def-prop}[theorem]{Definition-Proposition}

\newtheorem{remark}[theorem]{Remark}
\newtheorem{example}[theorem]{Example}

\newtheorem*{acknowledgment}{Acknowledgments}
\newtheorem{notation}[theorem]{Notation}

\newcommand{\ZZ}{{\mathbb Z}}
\newcommand{\NN}{{\mathbb N}}

\newcommand{\RR}{{\mathbb R}}

\newcommand{\kk}{{\mathbbm k}}

\newcommand{\floor}[1]{\lfloor #1 \rfloor}

\def\C{{\mathcal C}}

\def\R{{\mathcal R}}

\def\R{{\mathcal R}}

\def\1{{\bf 1}}
\def\0{{\bf 0}}

\begin{document}
	
	\title{Graded Families of ideals and Convex Regions}
	
	\author{Haoxi Hu}
	\address{Tulane University, Department of Mathematics and Statistics,
		6823 St. Charles Avenue
		New Orleans, LA 70118, USA}
	
	\email{hhu5@tulane.edu}

	
	\begin{abstract}
We study the interplay between \emph{graded families of ideals} in $\kk$-domains and their associated convex regions. These regions, called \emph{Newton-Okounkov regions}, arise naturally from \emph{graded families of ideals} associated to a \emph{valuation with one-dimensional leaves}. Our main focus is to compute \emph{asymptotic resurgence number} of a pair of \emph{graded families of ideals}. By combining techniques from \emph{Attouch--Wets topology} and convex-geometric properties of \emph{Newton-Okounkov regions}, we characterize the \emph{asymptotic resurgence number} through containment relations between the pair of corresponding \emph{Newton-Okounkov regions}.
	\end{abstract}
	\maketitle
	\section{Introduction}
	
Given a sequence of ideals $\{ \mathfrak{a}_{1}, \mathfrak{a}_{2}, \dots \}$ of a ring $R$, denoted as $\mathfrak{a}_{\bullet}$, then it is a \emph{graded family of ideals} if $\mathfrak{a}_{i} \cdot \mathfrak{a}_{j} \subseteq \mathfrak{a}_{i+j}$ for all $i,j \geq 1$. In particular, the \emph{graded family of ideals} is a \emph{filtration} if $\mathfrak{a}_{n+1} \subseteq \mathfrak{a}_{n}$ for all $n \geq 1$. Let $\mathfrak{a}_{\bullet}$ and $\mathfrak{b}_{\bullet}$ be graded families of ideals in a Notherian commutative ring $R$. \cite{HA2026} firstly defined the \emph{resurgence number} and the \emph{asymptotic resurgence number} of a pair of families of ideals as follows.

\vspace{2mm}

1. \emph{Resurgence number of the pair $(\mathfrak{a}_{\bullet}, \mathfrak{b}_{\bullet})$}:
\begin{align*}
\rho(\mathfrak{a}_{\bullet}, \mathfrak{b}_{\bullet}) =  \sup \left\{ \frac{s}{r} ~|~ \mathfrak{a}_{s} \not\subseteq \mathfrak{b}_{r}, s \geqslant 1, r \geqslant 1 \right\}
\end{align*}

2. \emph{Asymptotic resurgence number of the pair $(\mathfrak{a}_{\bullet}, \mathfrak{b}_{\bullet})$}: 
\begin{align*}
\hat{\rho}(\mathfrak{a}_{\bullet}, \mathfrak{b}_{\bullet}) = \sup \left\{ \frac{s}{r} ~|~ \mathfrak{a}_{st} \not\subseteq \mathfrak{b}_{rt}, s \geqslant 1, r \geqslant 1, t \gg 0 \right\}
\end{align*}

\vspace{2mm}

These definitions are inspired by \emph{ideal containment problems} that investigate when symbolic powers of an ideal are contained in its ordinary powers. After notable work done by Ein, Lazarsfeld and Smith \cite{EinLazarsfeldSmith} and Hochster and Huneke \cite{HochsterHuneke}, \emph{ideal containment problems} grew rapidly to an active field of research, and \emph{resurgence number} and the \emph{asymptotic resurgence number} were introduced as essential measures for non-containment between symbolic powers and ordinary powers of an ideal (cf. \cite{1,2,BocciHarbourne2010,3,4,5,6,7,8,9,VILLARREAL2023103656}). However, \emph{resurgence} and \emph{asymptotic resurgence} are very difficult to compute, only few classes of ideals are known. For instance, \emph{resurgence number} of \emph{star configuration of hyperplanes} (see \cite{BocciHarbourne2010}). 

In recent studies, Villarreal provides an understanding of \emph{resurgence number} through polyhedron \cite[Theorem 3.7]{VILLARREAL2023103656}: The inverse of the \emph{resurgence number} of a squarefree monomial ideal $I$ is equal to minimal inner product of vertices, such that one comes from \emph{covering polyhedra} of $I$ and another one is from its \emph{Alexander dual}. Inspired by Villarreal's idea, \cite{rncb} studies \emph{asymptotic resurgence number of a pair of graded families of monomial ideals} via associated convex regions constructed from \emph{Newton-Okounkov bodies} (see \cite{nob}). The original definition was first introduced by Kiumars Kaveh and A.G. Khovanskii in \cite{bodiesmultiplicity}, and it was called \emph{$\C$-convex region} and has been studied more recently (cf.   \cite{MultiplicitiesAT,CUTKOSKY201455,bodiesmultiplicity,ntbr,nob,rncb}). In this paper, we give the definition of \emph{$\C$-convex region} in different perspective, and call it \emph{Newton-Okounkov region of a graded family of ideals} (see \Cref{definition of NO region}). Our goal is to generalize \cite[Theorem 2.7]{rncb}: Given a pair of \emph{graded families of monomial ideals} $(\mathfrak{a}_{\bullet}, \mathfrak{b}_{\bullet})$ of a polynomial ring over a field, if the \emph{associated Rees algebra} of $\mathfrak{b}_{\bullet}$ is Noetherian, then $\hat{\rho}(\mathfrak{a}_{\bullet}, \mathfrak{b}_{\bullet}) = \sup \{ \lambda > 0 ~|~ \lambda \cdot \Gamma(\mathfrak{a}_{\bullet}) \not\subseteq \Gamma(\mathfrak{b}_{\bullet}) \}$. $\Gamma(\mathfrak{a}_{\bullet})$ is the \emph{Newton-Okounkov region} of $\mathfrak{a}_{\bullet}$, similarly for $\Gamma(\mathfrak{b}_{\bullet})$. 

We generalize \cite[Theorem 2.7]{rncb} to a broader setting by considering initial ideals associated with a \emph{valuation with one-dimensional leaves} on a Noetherian $\kk$-domain, rather than monomial ideals in a polynomial ring over a field. Moreover, we extend the theorem in two independent directions. One focuses on \emph{Newton-Okounkov region}, \textbf{Result 1} has no restrictions on associated Rees algebra while requires \emph{shortest distance} between boundaries of the pair of associated \emph{Newton-Okounkov regions} is not zero. The other direction extends \cite[Theorem 2.7]{rncb} directly from monomial ideals to initial ideals in the new settings, which leads to our \textbf{Result 2}.

\begin{introthm*}[\Cref{theorem: resurgence number = lambda}]
Let $R$ be a Noetherian $\kk$-domain, and define a valuation $\nu$ with one-dimensional leaves on $R$, and $S(R) = \ZZ^n_{\geq 0}$. Given a pair of graded families of ideals $(\mathfrak{a}_{\bullet}, \mathfrak{b}_{\bullet})$ of $R$, and the asymptotic resurgence number $\hat{\rho}(in_{\nu}(\mathfrak{a}_{\bullet}), in_{\nu}(\mathfrak{b}_{\bullet}))$ exists, then
\begin{align*}
\hat{\rho}(in_{\nu}(\mathfrak{a}_{\bullet}), in_{\nu}(\mathfrak{b}_{\bullet})) & \geq \inf \{ \lambda > 0 ~|~ \lambda \cdot \Gamma(\mathfrak{a}_{\bullet}) \subseteq \Gamma(\mathfrak{b}_{\bullet}) \} \\ 
& = \sup \{ \lambda > 0 ~|~ \lambda \cdot \Gamma(\mathfrak{a}_{\bullet}) \not\subseteq \Gamma(\mathfrak{b}_{\bullet}) \}
\end{align*}

If for any $\lambda > \inf \{ \lambda > 0 ~|~ \lambda \cdot \Gamma(\mathfrak{a}_{\bullet}) \subseteq \Gamma(\mathfrak{b}_{\bullet}) \}$, we have $\inf \{ d(x_1, x_2) ~|~ x_1 \in \lambda \Gamma(\mathfrak{a}_{\bullet}), x_2 \in \partial \Gamma(\mathfrak{b}_{\bullet}) \} > 0$ where $\partial \Gamma(\mathfrak{b}_{\bullet})$ is the boundary of $\Gamma(\mathfrak{b}_{\bullet})$, then we have the following equality:
\begin{align*}
\hat{\rho}(in_{\nu}(\mathfrak{a}_{\bullet}), in_{\nu}(\mathfrak{b}_{\bullet})) &= \inf \{ \lambda > 0 ~|~ \lambda \cdot \Gamma(\mathfrak{a}_{\bullet}) \subseteq \Gamma(\mathfrak{b}_{\bullet}) \} \\
&= \sup \{ \lambda > 0 ~|~ \lambda \cdot \Gamma(\mathfrak{a}_{\bullet}) \not\subseteq \Gamma(\mathfrak{b}_{\bullet}) \}
\end{align*}

\end{introthm*}

The containment of \emph{Newton-Okounkov regions} always exists but \emph{asymptotic resurgence number} does not always exist, see \Cref{eg no asy res num}. We investigate the existence of the \emph{asymptotic resurgence number}. We show that it exists whenever the \emph{shortest distance} from the \emph{Newton--Okounkov region} to the boundary of $\RR_{\geq 0}^n$ is positive, see \Cref{lemma: distance of two regions}. It follows that \Cref{thm: res num = inf NO} presents a different formulation of \textbf{Result 1} that does not assume the existence of the \emph{asymptotic resurgence number}.

\begin{introthm*}[\Cref{res = region in Notherian case}]
Let $R$ be a Noetherian $\kk$-domain, and define a valuation $\nu$ with one-dimensional leaves on $R$. Given a pair of graded families of ideals $(\mathfrak{a}_{\bullet}, \mathfrak{b}_{\bullet})$ of $R$, if the Rees algebra $\mathcal{R}(in_{\nu}(\mathfrak{b}_{\bullet}))$ is Noetherian, then
\begin{align*}
\hat{\rho}(in_{\nu}(\mathfrak{a}_{\bullet}), in_{\nu}(\mathfrak{b}_{\bullet})) &= \inf \{ \lambda > 0 ~|~ \lambda \cdot \Gamma(\mathfrak{a}_{\bullet}) \subseteq \Gamma(\mathfrak{b}_{\bullet}) \} \\
&= \sup \{ \lambda > 0 ~|~ \lambda \cdot \Gamma(\mathfrak{a}_{\bullet}) \not\subseteq \Gamma(\mathfrak{b}_{\bullet}) \}
\end{align*}
\end{introthm*}

\emph{Noetherianity} of associated Rees algebra is the main difference between these two results. \textbf{Result 2} requires \emph{Noetherianity}, so we summarize the classes of families of ideals admitting the Noetherian Rees algebras as follows. We assume $R$ is a Noetherian ring and $I$ is an ideal of $R$, and $\R$ is the notation for Rees algebra.
\begin{itemize}
\item[(1)] $\R(I)$, the Rees algebra associate to \emph{$I$-adic filtration}.
\item[(2)] $\R(\mathfrak{a}_{\bullet})$, the Rees algebra associated to  $\mathfrak{a}_{\bullet}$ where $\mathfrak{a}_{\bullet}$ is an \emph{$I$-stable filtration} (or \emph{$I$-good filtration}), see \Cref{stable filtration}.
\item[(3)] $\R(\bar{I})$, the Rees algebra associated to \emph{integral closure filtration of $I$}, and $R$ is an analytically unramified Noetherian local ring, see \Cref{integral closure filtration}.
\end{itemize}
There are plenty of classes of families of ideals may not admitting \emph{Noetherianity} on associated Rees algebra. For an example, consider the graded family of \emph{symbolic powers of $I$}, its associated Rees algebra is usually not Noetherian. Then \textbf{Result 1} can be applied.

By the nature of the \emph{valuation}, \emph{Newton-Okounkov regions} develop the invariant to \emph{initial ideals} but not to the corresponding ideals directly, so we introduce a new concept called \emph{compatibility condition} (see \Cref{compatibility}), allowing us to pass \textbf{Result 1} and \textbf{Result 2} from \textbf{initial ideals} to \textbf{all ideals}. The paper provides plenty classes of ideals that satisfy the \emph{compatibility condition} as follows.
\begin{itemize}
\item[(1)] Consider $R$ (a $\kk$-algebra and a domain) as graded algebra associated to a \emph{valuation with one-dimensional leaves}, then \emph{families of homogeneous ideals} of $R$ satisfy \emph{compatibility condition}, see \Cref{homogenous ideal and initial ideal}.
\item[(2)] \emph{families of symbolic powers of monomial ideals}, see \Cref{symbolic power}. 
\item[(3)] \emph{families of determinantal ideals}, see \Cref{determinantal ideal}.
\end{itemize}

A key feature of our paper is the new methodology developed in the proof of the main results, whereby we establish essential tools bridging algebraic properties of family of ideals with convex geometry. We begin by studying the convex-geometric properties of \emph{Newton-Okounkov regions}. In particular, when $R$ is Noetherian, each member of a \emph{graded family of ideals} corresponds to a polyhedron (called \emph{valued polyhedron}), and its \emph{recession cone} is the convex hull of the \emph{semigroup} of $R$. Then we work with \emph{Attouch-Wets topology}, and construct a sequence of \emph{valued polyhedra} that is \emph{Attouch-Wets convergent} to the \emph{Newton-Okounkov region} (see \Cref{prop: NP Attouch-Wets converges to NR}). This sequence of \emph{valued polyhedra}, along with convex-geometric properties of \emph{valued polyhedron} (see \Cref{properties of NO regions}), enable us to control the \emph{distance} between boundaries of the pair of associated \emph{Newton-Okounkov regions} and detect \emph{ideal containment} simultaneously.

\vspace{2mm}

\begin{acknowledgment}
We are thankful to Professor Tai Huy Ha for helpful discussions and suggestions.
\end{acknowledgment}

\vspace{2mm}

\section{Valuations, Newton-Okounkov regions, and Attouch-Wets topology}\label{section 2}

\vspace{2mm}

\begin{definition}
Let us recall that a \emph{valuation} on a domain $R$ is a function $\nu: K \rightarrow G$ where $G$ is an ordered abelian group and $K$ is the field of fraction of $R$, it satisfies the following axioms:
\begin{itemize}
\item[(1)] $\nu(0) = + \infty$.
\item[(2)] $\nu(xy) = \nu(x) + \nu(y)$ for $x,y \in K$.
\item[(3)] $\nu(x+y) \geq \min \{ \nu(x),\nu(y) \}$ for $x,y \in K$.
\end{itemize}

In particular, our $R$ is a $\kk$-domain, we restrict the valuation with more conditions:
\begin{itemize}
\item[(4)] $G$ is totally ordered.
\item[(5)] $\nu(k) = 0$ for all $k \in \kk$.
\item[(6)] $\nu(x) \geq 0$ for all $x \in R$.
\end{itemize}

\end{definition}

\begin{notation}
Throughout this paper, we define and use the following notations unless otherwise stated.
\begin{itemize}
\item $R$ : a $\kk$-algebra and a domain.
\item $\nu$ : \emph{valuation} on $R$. We assume $G$ is totally ordered $\ZZ^n$.
\item $S(R)$ : \emph{valuation semigroup} of $R$, i.e. $S(R) = \{ \nu(f) ~|~ f \neq 0 \in R \}$.
\item $F_{\succeq \alpha}$ : a \emph{$\nu-$filtration} on $R$ with $\alpha$. $F_{\nu \succeq \alpha} = \{ f \in R^{\times} ~|~ \nu(f) \succeq \alpha  \} \cup \{0\}$ where $\alpha \in G$.
\item  $F_{\succ \alpha}$ : it is defined similarly.
\end{itemize}
\end{notation}

\begin{definition}
Let $gr_{\nu}(R) = \oplus_{\alpha \in G} F_{\succeq \alpha} / F_{\succ \alpha}$, it is called \emph{Associated graded $R$} with \emph{valuation} $\nu$. Observe that every \emph{valuation} can induce a \emph{associated graded $R$}, and it is isomorphic to $R$ if it is graded by totally ordered abelian group (not necessary $\ZZ^n$). 

We can define the \emph{initial form of $f \in R$} associated to a \emph{valuation $\nu$} and \emph{initial ideal of $I$} associated to a \emph{valuation $\nu$} in $R$, denoted as $in_{\nu}(f)$ and $in_{\nu}(I)$ respectively, as follows:
\begin{align*}
in_{\nu}(f) &= f ~ \text{mod} ~ F_{\succ \alpha} ~ \text{where} ~ \alpha ~ \text{is the largest number such that} ~ f \in F_{\succeq \alpha} ~ \\ 
in_{\nu}(I) &= \, <in_{\nu}(f) ~|~ f \in I> 
\end{align*}

Observe that $\nu(f) = \alpha$. Indeed, we can consider $f$ as a finite sum of homogeneous components of $gr_{\nu}(R)$, and third axiom of \emph{valuation} forces $\nu(f) = \alpha$. The most famous example of such valuation is \emph{Gr\"obner valuation}. From now on, we fix $\alpha$ to be $\nu(f)$. 

Given a family of ideals $\mathfrak{a}_{\bullet} = \{ \mathfrak{a}_1, \, \mathfrak{a}_2, \, \mathfrak{a}_3, \, ... \,  \}$ of $R$, We can also define its \emph{initial family of of ideals associated to a valuation $\nu$} as follows:
\begin{align*}
in_{\nu}(\mathfrak{a}_{\bullet}) = \{ in_{\nu}(\mathfrak{a}_1), \, in(\mathfrak{a}_2), \, in_{\nu}(\mathfrak{a}_3), \, ... \, \}
\end{align*}

We say the valuation $\nu$ has \emph{one-dimensional leaves} if for every $\alpha \in G$ the quotient vector space $F_{\succeq \alpha} / F_{\succ \alpha}$ is at most 1-dimensional as a $\kk$-vector space.

\end{definition}

\begin{example}
$\nu$-filtration and $gr(R)$ can ne formulated in the similar construction in \cite[section 5]{CommutativeAlgebra}, which uses language of filtration of modules. In particular, the filtration is defined as direct sum of $M_i / M_{i+1}$ where $M_i \supset M_{i+1}$ for all $i \in \ZZ_{\geq 0}$, and these modules are called \emph{associated graded modules}.
\end{example}

\begin{example}\label{Groebner valuation}
Define \emph{Gröbner valuation} $\nu$ on $R = \kk[[x_1, \, x_2, \, ... \, x_n]]$ with totally ordered $\ZZ^n_{\geq 0}$. Let $f$ be the element of $\kk[[x_1, \, x_2, \, ... \, x_n]]$, then define $\nu(f) = \min \{ (a_1,\,a_2,\,...\,a_n) ~|~ cx_1^{a_1}x_2^{a_2}...x_n^{a_n} $ is a term of $f \}$. Similarly, we can also define \emph{Gröbner valuation} on polynomial ring over a field. It is obvious that the \emph{Gröbner valuation} has \emph{one-dimensional leaves}, and $S(R)$ generates $\ZZ^n_{\geq 0}$.

\end{example}

The graded algebra $gr_{\nu}(R)$ is often discussed with \emph{Khovanskii bases}. Recall that \emph{Khovanskii basis} is a set $\mathcal{B} \subset R$ such that it generates $gr_{\nu}(R)$. The notable paper \cite{khovanskii} shows plenty of properties of \emph{Khovanskii basis}. \cite[Algorithm 2.18]{khovanskii} gives an algorithm to compute the finite \emph{Khovanskii basis} via finite generating set of $R$. Furthermore, \cite[Lemma 3]{khovanskii} states \emph{Khovanskii basis} coincides algebra generating set under some conditions. These will make computation of the results of this paper possible if we assume $R$ is finitely generated.

After formally setting up the valuation part, we are ready to define the main subject. The \emph{Newton-Okounkov region} was first defined in \cite{bodiesmultiplicity}, where it was defined for the families of $\mathfrak{m}$-primary ideals on the Noetherian local domain of Krull dimension $n$ over a field $\kk$, and with maximal ideal $\mathfrak{m}$. In this paper, we will introduce \emph{Newton-Okounkov regions} through the language of rings and valuations.

\begin{definition}\label{definition of NO region}

Given a \emph{valuation $\nu$} on $R$ and a graded family $\mathfrak{a}_{\bullet}$ of nonzero ideals in $R$, define \emph{valued region of} $\mathfrak{a}_{k}$ as the following:
\vspace{1mm}
\begin{align*}
\text{VR}(\mathfrak{a}_{k}) = \text{conv} \{ \nu(f) ~|~ f \in \mathfrak{a}_k \backslash \{0\} \}
\end{align*}
\vspace{1mm}
Then we define the associated \emph{Newton-Okounkov region} of $\mathfrak{a}_{\bullet}$ to be
\vspace{1mm}
\begin{align*}
\Gamma(\mathfrak{a}_{\bullet}) = \overline{\bigcup_{k \geq 1} \frac{1}{k} \text{VR}(\mathfrak{a}_{k}) }
\end{align*}
\vspace{1mm}
In particular, if $R$ is Noetherian, \emph{valued region} is a polyhedron (see \Cref{properties of NO regions}(4)), then we call it \emph{valued polyhedron}, denoted as $\text{VP}$. 
\end{definition}

We have seen that the \emph{valuation} of \emph{initial form} of $f$ is same as the \emph{valuation} of $f$ for any $f \in R$, so the $\text{VR}(\mathfrak{a}) = \text{VR}(in_{\nu}(\mathfrak{a}))$ for any $\mathfrak{a}\in \mathfrak{a}_{\bullet}$ and $\Gamma(\mathfrak{a}_{\bullet}) = \Gamma(in_{\nu}(\mathfrak{a}_{\bullet}))$. Therefore we will focus on the \emph{initial graded family of ideals}. Before we state the properties of \emph{valued region} and \emph{Newton-Okounkov region}, we need to introduce some basic facts of polyhedron.

\begin{definition}
Let $A$ be a convex set, the \emph{recession cone} of $A$ is the set $\text{rec}(A) = \{ y \in A ~|~ \forall x \in A, \, x + y \in A \}$. In other words, \emph{recession cone} is the direction of an unbounded convex set when it extends to infinity. When $A$ is bounded, the \emph{recession cone} is simply zero.
\end{definition}

\begin{theorem}\cite[Minkowski-Weyl Theorem]{Convexity} \label{Minkowski-Weyl Theorem}
Let $P$ be the polyhedron, and let $V$ be the set of \emph{extreme points} of $P$, then $P = conv(V) + rec(P)$.
\end{theorem}

\begin{lemma}\label{properties of NO regions}
Given a graded family of ideals $\mathfrak{a}_{\bullet}$ of $R$ and a valuation $\nu$. We have the following properties:
\begin{itemize}
\item[(1)] $in(\mathfrak{a}_{\bullet})$ is graded.
\item[(2)] If $\nu$ has one-dimensional leaves, then initial ideal is uniquely determined by its valued region. Furthermore, one initial ideal contains another initial ideal if only if the corresponding valued region contains another one.
\item[(3)] $\frac{1}{k} \text{VR}(\mathfrak{a}_k) \subseteq \frac{1}{tk} \text{VR}(\mathfrak{a}_{tk})$ for all $t \geq 1$, likewise for its initial family of ideals.
\item[(4)] If $R$ is Noetherian, valued region is a polyhedron (called valued polyhedron $VP$), and $\text{VP}(\mathfrak{a}) = conv(V) + conv(S(R))$ for any $\mathfrak{a} \in \mathfrak{a}_{\bullet}$, and $V$ is its set of extreme points. 
\item[(5)] If $R$ is Noetherian, and $S(R)$ generates $\ZZ^n_{\geq 0}$, then $\text{VP}(\mathfrak{a}) = conv(V) + \RR^n_{\geq 0}$ for any $\mathfrak{a} \in \mathfrak{a}_{\bullet}$, and $V$ is the set of extreme points. Furthermore, there exists a real number $r$ such that the boundary of $\Gamma(\mathfrak{a}_{\bullet})$ approaches or intersects the boundary of $(r+\RR^n_{\geq 0})$. In other words, $d(\partial (r+\RR^n_{\geq 0}), A) = 0$ where $A = \{ y ~|~ y \in \partial \Gamma(\mathfrak{a}_{\bullet}) \; \text{and} \; |y| \gg 0 \}$
\end{itemize}

\end{lemma}

\begin{proof}
(1) We know that $\mathfrak{a}_i \cdot \mathfrak{a}_j \subseteq \mathfrak{a}_{i+j}$ for all $i,j$. Take $x \in \mathfrak{a}_i$ and $y \in \mathfrak{a}_j$, then $in_{\nu}(x) \in F_{\succeq \nu(x)} / F_{\succ \nu(x)}$ and $in_{\nu}(y) \in F_{\succeq \nu(y)} / F_{\succ \nu(y)}$. Therefore $in_{\nu}(x)$ is homogeneous component of $x$ which has same valuation of $x$ on $gr_{\nu}(R)$. Note that $in_{\nu}(x)$ is a lowest degree component with respect to $\nu$, and similarly for $in_{\nu}(y)$. Then $in_{\nu}(x) \cdot in_{\nu}(y) \in  F_{\succeq \nu(xy)}$, is exactly $in_{\nu}(xy) = xy ~ \text{mod} ~ F_{\succ \nu(xy)}$, so $in_{\nu}(\mathfrak{a}_{\bullet})$ is graded.

(2) By the definition of \emph{one-dimensional valuation}, homogeneous component of $\alpha \in \ZZ^n$ is unique up to scalars $\kk$, so different initial ideals corresponds to different sets of points in $\ZZ^n$, because initial ideal is generated by homogeneous components. The second statement follows from the fact that one initial ideal contains another initial ideal if one ideal contains all homogeneous components of another initial ideal.

(3) Since $t \mathfrak{a}_k = \mathfrak{a}_k^t \subseteq \mathfrak{a}_{tk}$ for all $t \geq 1$, then $\frac{1}{k} \text{VR}(\mathfrak{a}_k) \subseteq \frac{1}{tk} \text{VR}(\mathfrak{a}_{tk})$ for all $t \geq 1$.

(4) Since $R$ is Noetherian, every ideal is finitely generated, so is its initial ideal, say $in_{\nu}(\mathfrak{a})$. Following (2), the set of generators of $in_{\nu}(\mathfrak{a})$ gives a set of points in $\ZZ^n$, denoted as $V$, then the convex hull of $V$ is a polytope. By the definition of \emph{valued region}, we can write $\text{VR}(\mathfrak{a}) = conv(V) + conv(S(R))$. Now by \Cref{Minkowski-Weyl Theorem}, it is a \emph{polyhedron}.

(5) $\text{VP}(\mathfrak{a}) = conv(V) + \RR^n_{\geq 0}$ where $V$ is the finite set of \emph{extreme points} by (4). Then we can observe that for any $y \in \Gamma(\mathfrak{a}_{\bullet})$, $y + \RR^n_{\geq 0}$ is also contained in $\Gamma(\mathfrak{a}_{\bullet})$. Define $r_i = \inf \{ y_i ~|~ y = (y_1,y_2, ... \, ,y_n) \in \Gamma(\mathfrak{a}_{\bullet}) \}$, and $r = (r_1,r_2, ... \, ,r_n)$, then $d(\partial (r+\RR^n_{\geq 0}), A) = 0$ where $A = \{ y ~|~ y \in \partial \Gamma(\mathfrak{a}_{\bullet}) \; \text{and} \; |y| \gg 0 \}$

\end{proof}

We can observe that \emph{valued regions} and \emph{Newton-Okounkov regions} are unbounded convex sets. To study the properties of unbounded convex sets, we need a new method based on \emph{Attouch-Wes Topology}. The paper by Gerald Beer (see \cite{AttouchWets}) provides a concise introduction to the \emph{Attouch--Wets topology} and its fundamental properties, which are essential in the analysis of unbounded convex sets. Note that all information of \emph{Attouch-Wes Topology} comes from this paper.

\begin{definition}\label{definition: set-up for topology}

Let $(X,d)$ be the \emph{Topology $X$ of Hausdorff distance $d$}, for nonempty subsets $A, B \subseteq X$ and $\epsilon > 0$, define
\begin{align*}
A^{\epsilon} = \{ x \in X : d(x,A) < \epsilon \} \; \text{and} \; e_d(B,A) = \sup_{b \in B} d(b,A) = \inf \{ \epsilon > 0 : B \subseteq A^{\epsilon} \}
\end{align*}

\end{definition}

\begin{definition}\label{definition: Attouch-Wets Topology}
Let $\mathscr{B}_d$ is a collection of all open balls with fixed center at $0$ with integral radius. A sequence of sets $\{ A_i \}_{i \in \NN}$ in power sets of $X$ is called \emph{Attouch-Wets convergence} to an nonempty subset $A \subseteq X$ if for all $B \in \mathscr{B}_d$ and for all $\epsilon > 0$, we have the both followings for $i$ big enough.
\begin{align*}
A \cap B \subseteq A_{i}^{\epsilon} \; \text{and} \; A_i \cap B \subseteq A^{\epsilon}
\end{align*}
When this occurs, we will write $A = AW_d - \lim A_i$.

There is an \emph{equivalent} definition for \emph{Attouch-Wets convergence}, the followings are needed to be satisfied for each $B \in \mathscr{B}_d$:
\begin{align*}
\lim_{i \in \NN} e_d(A \cap B, A_i) = 0 \; \text{and} \; \lim_{i \in \NN} e_d(A_i \cap B, A) = 0
\end{align*}

\end{definition}

\begin{lemma}\label{lemma: finite properties of newton okounkov region}
Let $\mathfrak{a}_{\bullet}$ be the graded family of ideals of $R$, then $\Gamma(\mathfrak{a}_{\bullet}) = \overline{ \cup_{k \geq t_0} \frac{1}{k} \text{VR}(\mathfrak{a}_k)}$ for a positive integer $t_0$. In particular, $k$ can be replaced by $ak$ where $a$ is also a positive integer.

\end{lemma}

\begin{proof}

Following \Cref{properties of NO regions}, for any $\frac{1}{k} \text{VR}(\mathfrak{a}_k)$ such that $k < t_0$, we can always find an integer $b_k k \geq t_0$ where $b_k$ is also an integer, then $\frac{1}{k} \text{VR}(\mathfrak{a}_k) \subseteq \frac{1}{b_k k} \text{VR}(\mathfrak{a}_{b_k k})$. Therefore $\Gamma(\mathfrak{a}_{\bullet}) = \overline{ \cup_{k \geq t_0} \frac{1}{k} \text{VR}(\mathfrak{a}_k)}$. Now replacing $k$ by $ak$. Similarly, for any $t_1$ that is not the multiple of $a$, there always exists $\frac{1}{t_1} \text{VR}(\mathfrak{a}_{t_1}) \subseteq \frac{1}{a t_1 t_0} \text{VR}(\mathfrak{a}_{a t_1 t_0})$.

\end{proof}

\begin{proposition}\label{prop: NP Attouch-Wets converges to NR}
Define the sequence of valued regions of any graded family of ideals as follows
\begin{align*}
\{ B_t \}_{t \in \mathbb{Z}_{\geq 1}} \; \text{where} \; B_t = \frac{1}{t!} \text{VR}(\mathfrak{b}_{t!})
\end{align*}
Then $\Gamma(\mathfrak{b}_{\bullet}) = AW_d - \lim B_t$. In particular, the same statement is also true for any tail sequence of $\{ B_t \}$.
\end{proposition}

\begin{proof}

We will show that $\Gamma(\mathfrak{b}_{\bullet}) = AW_d - \lim B_t$ via the equivalent definition of \emph{Attouch–Wets convergence} (see \Cref{definition: Attouch-Wets Topology}). So each $B \in \mathscr{B}_d$, we need to show that both $\lim_{t} e_d(\Gamma(\mathfrak{b}_{\bullet}) \cap B, B_t) = 0$ and $\lim_{t} e_d(B_t \cap B, \Gamma(\mathfrak{b}_{\bullet})) = 0$, the latter is trivial since $B_t \subseteq \Gamma(\mathfrak{b}_{\bullet})$ for all $t$.

Observe that $\{ B_t \}_{t \in \mathbb{Z}_{\geq 1}}$ converges to $B' = \cup_{t \in \mathbb{Z}_{\geq 1}} B_t$ set-theoretically, and $\overline{B'} = \Gamma(\mathfrak{b}_{\bullet})$. Indeed, following Lemma \ref{lemma: finite properties of newton okounkov region}, take any $\frac{1}{t} \text{VR}(\mathfrak{b}_{t}) \subseteq \Gamma(\mathfrak{b}_{\bullet})$, we will have $\frac{1}{t} \text{VR}(\mathfrak{b}_{t}) \subseteq \frac{1}{t!} \text{VR}(\mathfrak{b}_{t!})$, and  $\{ B_t \}_{t \in \mathbb{Z}_{\geq 1}}$ is a non-decreasing sequence.

since $\Gamma(\mathfrak{b}_{\bullet}) = \overline{\cup_{t \geq 1} B_t}$, then for any point $x \in \Gamma(\mathfrak{b}_{\bullet})$, we can always find a sequence $\{ x_t \}_{t \geq 1}$ where $x_t \in B_t$ and converges to $x$. Because distance from a point to $B$ is same as to the closure $\overline{B}$, i.e. $e_d(\Gamma(\mathfrak{b}_{\bullet}) \cap B, B_t) = e_d(\Gamma(\mathfrak{b}_{\bullet}) \cap \overline{B}, B_t)$, we can just consider $\overline{B}$. Given $\epsilon > 0$, for any $x \in \Gamma(\mathfrak{b}_{\bullet}) \cap \overline{B}$, there exists corresponding $N_x \in \mathbb{Z}$ such that $t \geq N_x$ and $y_x \in B_t$, the distance $d(x,y_x) < \frac{\epsilon}{2}$. Consider an open cover $\{ B(x,\frac{\epsilon}{2}) \}_{x \in \Gamma(\mathfrak{b}_{\bullet}) \cap B}$ of $\Gamma(\mathfrak{b}_{\bullet}) \cap \overline{B}$ where $B(x,\frac{\epsilon}{2})$ is an open ball centered at $x$ with radius $\frac{\epsilon}{2}$. since $\Gamma(\mathfrak{b}_{\bullet}) \cap \overline{B}$ is compact, then there exists a finite subcover, relabel them as $B(x_{i_1},\frac{\epsilon}{2}), B(x_{i_2},\frac{\epsilon}{2}), ... \, , B(x_{i_N},\frac{\epsilon}{2})$. Then for all $t \geq \sup \{ N_{x_{i_j}} : 1 \leq j \leq N \}$, we have $x \in B(x_{i_j},\frac{\epsilon}{2})$ for some $j$, and $B_t$ contains  $y_{x_{i_j}}$ for all $j$. Therefore $d(x,y_{x_{i_j}}) \leq d(x, x_{i_j}) + d(x_{i_j}, y_{x_{i_j}}) < \epsilon$, this implies $e_d(\Gamma(\mathfrak{b}_{\bullet}) \cap B, B_t) < \epsilon$,  so for any given $B \in \mathscr{B}_d$, $\lim_{t} e_d(\Gamma(\mathfrak{b}_{\bullet}) \cap B, B_t) = 0$, as we desire.

\end{proof}

\vspace{2mm}

\section{Asymptotic resurgence number of a pair of graded families of ideals}\label{section 3}

\vspace{2mm}

In this section we show main results that share the same goal: the \emph{asymptotic resurgence number} of a pair of families of ideals equals to the supremum of noncontainment of the corresponding pair of \emph{Newton-Okounkov regions}. 

We derive properties of \emph{valued regions} and \emph{Newton-Okounkov regions} in the previous section, which are \Cref{properties of NO regions} and \ref{prop: NP Attouch-Wets converges to NR}, they are the main tools for \Cref{theorem: resurgence number = lambda}. However, instead of focusing the general case, we want $R$ to be Noetherian, then \emph{valued regions} are just \emph{valued polyhedra}. We also show \Cref{lemma: distance of two regions}, leading to a cleaner and slightly restricted version of \Cref{theorem: resurgence number = lambda}, which is the \Cref{thm: res num = inf NO}. Moreover, we show \Cref{main tool for Noetherian case}, and this allows us to prove \Cref{res = region in Notherian case} that generalize \cite[Theorem 2.7]{rncb}.

We have seen that the \emph{Newton-Okounkov regions} and \emph{valued regions} only establish the invariant to initial ideals, not the ideals in general. To study the \emph{containment problem} of families of ideals, we need the following definition, then we can pass the \emph{containment} of a pair of \emph{graded families of initial ideals} to the \emph{containment} of a pair of \emph{graded families of ideals}.

\begin{definition}\label{compatibility}

Given a pair of families of ideals $(\mathfrak{a}_{\bullet}, \mathfrak{b}_{\bullet})$ of $R$, we say that $(\mathfrak{a}_{\bullet}, \mathfrak{b}_{\bullet})$ satisfies \emph{compatibility condition} associated to a valuation $\nu$, denoted as \emph{$(\text{Compat}_{\nu})$}, if for all positive integers $s,r$, we have
\begin{align*}
    \mathfrak{a}_{st} \not\subseteq \mathfrak{b}_{rt} \Rightarrow in_{\nu}(\mathfrak{a}_{st}) \not\subseteq in_{\nu}(\mathfrak{a}_{st}) \; \text{for all} \; t \gg 0
\end{align*}

As a result, if $(\mathfrak{a}_{\bullet}, \mathfrak{b}_{\bullet})$ satisfies \emph{$(\text{Compat}_{\nu})$}, we have
\begin{align*}
    \mathfrak{a}_{st} \not\subseteq \mathfrak{b}_{rt} \iff in_{\nu}(\mathfrak{a}_{st}) \not\subseteq in_{\nu}(\mathfrak{a}_{st}) \; \text{for all} \; t \gg 0
\end{align*}

\end{definition}

\begin{remark}\label{homogenous ideal and initial ideal}

If we consider $R$ as the graded algebra $gr_{\nu}(R)$, with the valuation $\nu$ with \emph{one-dimensional leaves}, then the homogeneous ideals of $gr_{\nu}(R)$ are the same as their \emph{initial ideals} since different homogeneous components correspond to different degree up to scalars. Therefore, if given a pair of graded families of \emph{homogeneous ideals} $(\mathfrak{a}_{\bullet}, \mathfrak{b}_{\bullet})$ of $R$, then the \emph{asymptotic resurgence} $\hat{\rho}(in_{\nu}(\mathfrak{a}_{\bullet}), in_{\nu}(\mathfrak{b}_{\bullet}))$ can be replaced by $\hat{\rho}(\mathfrak{a}_{\bullet}, \mathfrak{b}_{\bullet})$. 

\begin{example}\label{symbolic power}
Following \Cref{homogenous ideal and initial ideal}, the monomial ideals of polynomial ring over a field with \emph{Gröbner valuation} $\nu$ satisfy $(Compat_{\nu})$, see \Cref{Groebner valuation} for more details. To have the graded families, we can take the set of \emph{ordinary powers} of monomial ideals as our family, and we also can take \emph{symbolic powers}, since \emph{symbolic powers} of an ideal are still monomial ideals.
\end{example}

\end{remark}

\begin{example}\label{determinantal ideal}
\cite[Section 6.5]{gbca} gives the classic definition of \emph{determinantal ideals}. Given a ring $\kk[X]$, where the \emph{matrix of indeterminantes} is $X = (x_{i,j})_{i \in [n], j \in [m]}$, and the \emph{determinantal ideal} of $X$ is $I_t(X) = \, <\text{determinats of t-minors of} \, X>$. Now we want to extend $X$ to be the following:
\begin{align*}
X &=
\begin{pmatrix}
x_{1,1} & x_{1,2} & ... \, \\
x_{2,1} & x_{2,2} & ... \, \\
\vdots & \vdots   \\
x_{n,1} & x_{n,2} & ... \, \\
\end{pmatrix}
\end{align*}
Define the \emph{matrix of indeterminates} to be $A = (x_{i,j})$ for some $i$ and $j$, where size of $A$ is bounded. For an example, let $n = 5$, we can let $A$ to be
\begin{align*}
\begin{pmatrix}
x_{1,1} & x_{1,4} & x_{1,9} \\
x_{2,1} & x_{2,4} & x_{1,9} \\
x_{4,1} & x_{4,4} & x_{4,9}  \\
\end{pmatrix}
\end{align*}
Define a \emph{valuation} $\nu$ on $\kk[X]$ by $\nu(x_{i,j}) = g_{i,j} \in G$ where $g_{i,j}$ is a generator of $G$. Note that different generators of a determinantal ideal correspond to different $t$-minors, and different $t$-minors can be described by the \emph{diagonals}. Thus, fixing a \emph{diagonal monomial order} with respect to $\nu$, it will construct a one-to-one correspondence between the \emph{determinantal ideal} and its \emph{initial ideal}, see \cite[Section 6.5.2]{gbca} for more details. Given a pair of \emph{families of determinantal ideals $(\mathfrak{a}_{\bullet}, \mathfrak{b}_{\bullet})$} of $\kk[X]$, it satisfies $(Compat_{\nu})$ because of the one-to-one correspondence. In particular, when $X$ is finite, we can define the \emph{valuation} to $\ZZ^n$.
\end{example}

To state the theorems in the most general settings, we only shows for \emph{asymptotic resurgence number} of a pair of \emph{families of initial ideals}. However, \Cref{homogenous ideal and initial ideal}, \Cref{symbolic power}, and \Cref{determinantal ideal} carry the statements to all \emph{families of ideals} as we eventually want.

\begin{theorem}\label{theorem: resurgence number = lambda}
Let $R$ be a Noetherian $\kk$-domain, and define a valuation $\nu$ with one-dimensional leaves on $R$, and $S(R) = \ZZ^n_{\geq 0}$. Given a pair of graded families of ideals $(\mathfrak{a}_{\bullet}, \mathfrak{b}_{\bullet})$ of $R$, and the asymptotic resurgence number $\hat{\rho}(in_{\nu}(\mathfrak{a}_{\bullet}), in_{\nu}(\mathfrak{b}_{\bullet}))$ exists, then
\begin{align*}
\hat{\rho}(in_{\nu}(\mathfrak{a}_{\bullet}), in_{\nu}(\mathfrak{b}_{\bullet})) & \geq \inf \{ \lambda > 0 ~|~ \lambda \cdot \Gamma(\mathfrak{a}_{\bullet}) \subseteq \Gamma(\mathfrak{b}_{\bullet}) \} \\ 
& = \sup \{ \lambda > 0 ~|~ \lambda \cdot \Gamma(\mathfrak{a}_{\bullet}) \not\subseteq \Gamma(\mathfrak{b}_{\bullet}) \}
\end{align*}

If for any $\lambda > \inf \{ \lambda > 0 ~|~ \lambda \cdot \Gamma(\mathfrak{a}_{\bullet}) \subseteq \Gamma(\mathfrak{b}_{\bullet}) \}$, we have  $\inf \{ d(x_1, x_2) ~|~ x_1 \in \lambda \Gamma(\mathfrak{a}_{\bullet}), x_2 \in \partial \Gamma(\mathfrak{b}_{\bullet}) \} > 0$ where $\partial \Gamma(\mathfrak{b}_{\bullet})$ is the boundary of $\Gamma(\mathfrak{b}_{\bullet})$. Then we have the following equality:
\begin{align*}
\hat{\rho}(in_{\nu}(\mathfrak{a}_{\bullet}), in_{\nu}(\mathfrak{b}_{\bullet})) &= \inf \{ \lambda > 0 ~|~ \lambda \cdot \Gamma(\mathfrak{a}_{\bullet}) \subseteq \Gamma(\mathfrak{b}_{\bullet}) \} \\
&= \sup \{ \lambda > 0 ~|~ \lambda \cdot \Gamma(\mathfrak{a}_{\bullet}) \not\subseteq \Gamma(\mathfrak{b}_{\bullet}) \}
\end{align*}

\end{theorem}

\begin{proof}

Observe that $\sup \{ \lambda > 0 \, | \, \lambda \cdot \Gamma(\mathfrak{a}_{\bullet}) \not\subseteq \Gamma(\mathfrak{b}_{\bullet}) \} = \inf \{ \lambda > 0 \, | \, \lambda \cdot \Gamma(\mathfrak{a}_{\bullet}) \subseteq \Gamma(\mathfrak{b}_{\bullet}) \}$. Indeed, take any $\lambda = \frac{r}{s} > \sup \{ \lambda > 0 \, | \, \lambda \cdot \Gamma(\mathfrak{a}_{\bullet}) \not\subseteq \Gamma(\mathfrak{b}_{\bullet}) \}$, then $\lambda \Gamma(\mathfrak{a}_{\bullet}) \subseteq \Gamma(\mathfrak{b}_{\bullet})$, so $\sup \{ \lambda > 0 \, | \, \lambda \cdot \Gamma(\mathfrak{a}_{\bullet}) \not\subseteq \Gamma(\mathfrak{b}_{\bullet}) \} \leq \inf \{ \lambda > 0 \, | \, \lambda \cdot \Gamma(\mathfrak{a}_{\bullet}) \subseteq \Gamma(\mathfrak{b}_{\bullet}) \}$. Similarly, if we take any $\lambda = \frac{r}{s} < \inf \{ \lambda > 0 \, | \, \lambda \cdot \Gamma(\mathfrak{a}_{\bullet}) \subseteq \Gamma(\mathfrak{b}_{\bullet}) \}$, we will have $\sup \{ \lambda > 0 \, | \, \lambda \cdot \Gamma(\mathfrak{a}_{\bullet}) \not\subseteq \Gamma(\mathfrak{b}_{\bullet}) \} \geq \inf \{ \lambda > 0 \, | \, \lambda \cdot \Gamma(\mathfrak{a}_{\bullet}) \subseteq \Gamma(\mathfrak{b}_{\bullet}) \}$.

Note that $in_{\nu}(\mathfrak{a})$ contains $in_{\nu}(\mathfrak{b})$ if only if $\text{VP}(\mathfrak{a})$ contains $\text{VP}(\mathfrak{b})$ for any $\mathfrak{a},\, \mathfrak{b}$ ideals of $R$ by \Cref{properties of NO regions}(2). Let $\lambda = \frac{r}{s} > \hat{\rho}(in_{\nu}(\mathfrak{a}_{\bullet}),in_{\nu}( \mathfrak{b}_{\bullet}))$, and let $\alpha \in \text{VP}(\mathfrak{a}_{rt})$ for all $t \geq t_0$ if $t_0$ is large enough, then $\alpha \in \text{VP}(\mathfrak{b}_{st})$. Now we have
\begin{align*}
\frac{1}{rt} \lambda \alpha = \frac{1}{st} \alpha \in \frac{1}{st} \lambda \text{VP}(\mathfrak{b}_{st}) \subseteq \Gamma(\mathfrak{b}_{\bullet}) \; \text{and} \; \frac{1}{rt} \lambda \alpha \in \lambda \text{VP}(\mathfrak{a}_{rt}) \subseteq \lambda \Gamma(\mathfrak{a}_{\bullet})
\end{align*}

Since the above holds for all $t \geq t_0$, then by Lemma \ref{lemma: finite properties of newton okounkov region}, we can conclude that $\lambda \Gamma(\mathfrak{a}_{\bullet}) \subseteq \Gamma(\mathfrak{b}_{\bullet})$. It immediately follows that $\hat{\rho}(in_{\nu}(\mathfrak{a}_{\bullet}), in_{\nu}(\mathfrak{b}_{\bullet})) \geq \inf \{ \lambda > 0 \, | \, \lambda \cdot \Gamma(\mathfrak{a}_{\bullet}) \subseteq \Gamma(\mathfrak{b}_{\bullet}) \}$.

\vspace{3mm}

Let $\lambda = \frac{r}{s} > \inf \{ \lambda > 0 \, | \, \lambda \cdot \Gamma(\mathfrak{a}_{\bullet}) \subseteq \Gamma(\mathfrak{b}_{\bullet}) \}$, then define the two sequences of \emph{valued Polyhedra} as follows:
\begin{align*}
\{ A_t \}_{t \in \mathbb{Z}_{\geq 1}} \, \text{where} \, A_t = \frac{1}{rt} \text{VP}(\mathfrak{a}_{rt}) \, \text{and} \, \{ B_t \}_{t \in \mathbb{Z}_{\geq 1}} \, \text{where} \, B_t = \frac{1}{t!} \text{VP}(\mathfrak{b}_{t!})
\end{align*}

By Proposition \ref{prop: NP Attouch-Wets converges to NR}, $\Gamma(\mathfrak{b}_{\bullet}) = \overline{ \cup_{t \geq 1} B_t }$, and $\Gamma(\mathfrak{b}_{\bullet}) = AW_d - \lim B_t$. By \Cref{properties of NO regions}(5), fix $r$ to be the real number such that $\partial \Gamma(\mathfrak{b}_{\bullet})$ approaches or intersects $\partial (r+\RR^n_{\geq 0})$. Let $B_t = conv(V_t) + \RR^n_{\geq 0}$ where $V_t$ is the set of \emph{extreme points} of $B_t$, then $\partial B_t - \partial conv(V_t)$ approaches to $\partial (r+\RR^n_{\geq 0})$ parallelly as $t$ goes to infinity, so does to $\partial \Gamma(\mathfrak{b}_{\bullet})$.

Following the assumption, we have  $\inf \{ d(x_1, x_2) ~|~ x_1 \in \lambda \Gamma(\mathfrak{a}_{\bullet}), x_2 \in \partial \Gamma(\mathfrak{b}_{\bullet}) \} > 0$, then there exists $\epsilon > 0$ such that $(\lambda \Gamma(\mathfrak{a}_{\bullet}))^{\epsilon}$ is contained in a closed subset, denoted by $H_1$. Moreover, $(\Gamma(\mathfrak{b}_{\bullet})^{c})^{\epsilon}$ is contained in another closed subset, denoted by $H_2$, and we want $H_1$ and $H_2$ never intersect. In particular, $H_2^c$ strictly contains $\lambda \Gamma(\mathfrak{a}_{\bullet})$ and is strictly contained in $\Gamma(\mathfrak{b}_{\bullet})$. 

Let $t_1$ be large enough such that distance from $\partial B_t - \partial conv(V_t)$ to $\partial \Gamma(\mathfrak{b}_{\bullet})$ is less than $\frac{\epsilon}{2}$ for all $t \geq t_1$ because of the conclusion from the paragraph before the previous one. Note that for any $B_t$ the $conv(V_t)$ is a bounded convex set, and $B_t \subseteq B_{t+1}$ for all $t$. Following the definition of \emph{Attouch-Wets convergence}, we can let $t_0$ be big enough, such that for all $t \geq t_0 \geq t_1$, $\Gamma(\mathfrak{b}_{\bullet}) \cap C_d \subseteq B_t^{\frac{\epsilon}{2}}$ and $B_t \cap C_d \subseteq \Gamma(\mathfrak{b}_{\bullet})^{\frac{\epsilon}{2}}$ for all $d \leq d_0$, where $C_d \in \mathscr{B}_d$ and $C_{d_0}$ is the first ball that fully contains $conv(V_{t_1})$. Then $B_t \supsetneq H_2^c \supsetneq \lambda \Gamma(\mathfrak{a}_{\bullet})$ for all $t \geq t_0$. Now let $B_{k(t_0!)}' = \frac{1}{sk(t_0!)} \text{VP}(\mathscr{b}_{sk(t_0!)}) \supseteq B_{t_0} \supsetneq \lambda \Gamma(\mathfrak{a}_{\bullet})$ for all $k \geq 1$ by  \Cref{lemma: finite properties of newton okounkov region}. Obviously, $B_{k(t_0!)}' \supseteq A_{k(t_0!)}$ for all $k \geq 1$, then for all $t = k (t_0!)$, we have 
\begin{align*}
\frac{1}{st} \text{VP}(\mathfrak{b}_{st}) \supseteq \frac{1}{rt} \lambda \text{VP}(\mathfrak{a}_{rt}) \Rightarrow \frac{1}{st} \text{VP}(\mathfrak{b}_{st}) \supseteq \frac{1}{st} \text{VP}(\mathfrak{a}_{rt}) \Rightarrow \text{VP}(\mathfrak{b}_{st}) \supseteq \text{VP}(\mathfrak{a}_{rt})
\end{align*}
Thus we have $\lambda > \hat{\rho}(in_{\nu}(\mathfrak{a}_{\bullet}), in_{\nu}(\mathfrak{b}_{\bullet}))$, so $\hat{\rho}(in_{\nu}(\mathfrak{a}_{\bullet}), in_{\nu}(\mathfrak{b}_{\bullet})) \leq \inf \{ \lambda > 0 \, | \, \lambda \cdot \Gamma(\mathfrak{a}_{\bullet}) \subseteq \Gamma(\mathfrak{b}_{\bullet}) \}$
\end{proof}

\begin{lemma}\label{lemma: distance of two regions}
Let $R$ be a Noetherian $\kk$-domain, and define a valuation $\nu$ with one-dimensional leaves on $R$, and $S(R) = \ZZ^n_{\geq 0}$. Let $\mathfrak{a}_{\bullet}$ and $\mathfrak{b}_{\bullet}$ be graded families of ideals of $R$. Suppose $\inf \{ d(x_1, x_2) ~|~ x_1 \in \Gamma(\mathfrak{b}_{\bullet}), x_2 \in \partial \mathbb{R}_{\geq 0}^n \} > 0$, then for any $\lambda > \inf \{ \lambda > 0 ~|~ \lambda \cdot \Gamma(\mathfrak{a}_{\bullet}) \subseteq \Gamma(\mathfrak{b}_{\bullet}) \}$, it satisfies that $\inf \{ d(x_1, x_2) ~|~ x_1 \in \lambda \Gamma(\mathfrak{a}_{\bullet}), x_2 \in \partial \Gamma(\mathfrak{b}_{\bullet}) \} > 0$
\end{lemma}

\begin{proof}

We prove by contraction, suppose $\inf \{ d(x_1, x_2) ~|~ x_1 \in \lambda \Gamma(\mathfrak{a}_{\bullet}), x_2 \in \partial \Gamma(\mathfrak{b}_{\bullet}) \} = 0$, then there exist a number $\lambda_0 > \inf \{ \lambda > 0 ~|~ \lambda \cdot \Gamma(\mathfrak{a}_{\bullet}) \subseteq \Gamma(\mathfrak{b}_{\bullet}) \}$ satisfying the one of the following cases.
\begin{itemize}[leftmargin=15mm]
\item[Case 1:] $\lambda_0 \Gamma(\mathfrak{a}_{\bullet})$ intersects with boundary of $\Gamma(\mathfrak{b}_{\bullet})$ (i.e. $\partial \Gamma(\mathfrak{b}_{\bullet})$).

\item[Case 2:]$\lambda_0 \Gamma(\mathfrak{a}_{\bullet})$ approaches to the boundary of $\Gamma(\mathfrak{b}_{\bullet})$.
\end{itemize}

The second case is possible since $\Gamma(\mathfrak{a}_{\bullet})$ is an unbounded convex set. We show the second case, then the first case follows easily.

By the assumption, set $ d = \inf \{ d(x_1, x_2) ~|~ x_1 \in \Gamma(\mathfrak{b}_{\bullet}), x_2 \in \partial \mathbb{R}_{\geq 0}^n \} > 0$. Since $\lambda_0 \Gamma(\mathfrak{a}_{\bullet})$ approaches to the boundary of $\Gamma(\mathfrak{b}_{\bullet})$, we can find a point $x \in \Gamma(\mathfrak{a}_{\bullet})$ and a number $\lambda_1 < \lambda_0$ such that $d(\lambda_0 x, \partial \Gamma(\mathfrak{b}_{\bullet})) < \frac{\lambda_0 - \lambda_1}{\lambda_0} d$, and $\lambda_1 > \inf \{ \lambda > 0 ~|~ \lambda \cdot \Gamma(\mathfrak{a}_{\bullet}) \subseteq \Gamma(\mathfrak{b}_{\bullet}) \}$. Note that $\lambda_1 x + \mathbb{R}^n_{\geq 0} \subseteq \Gamma(\mathfrak{b}_{\bullet})$ by proof of \Cref{properties of NO regions}, then there exist an open ball centered at $\lambda_0 x$ with radius of $\frac{\lambda_0 - \lambda_1}{\lambda_0} d$ inside $\Gamma(\mathfrak{b}_{\bullet})$. Indeed, $d(\lambda_0 x , \partial \mathbb{R}_{\geq 0}^n) > d$  since $\lambda_0 x$ is contained in the interior of $\Gamma(\mathfrak{b}_{\bullet})$, then distance from $\lambda_0 x$ to the boundary of $\lambda_1 x + \mathbb{R}^n_{\geq 0}$ is greater than $\frac{\lambda_0 - \lambda_1}{\lambda_0} d$. However, it contradicts to the fact that $d(\lambda_0 x, \partial \Gamma(\mathfrak{b}_{\bullet})) < \frac{\lambda_0 - \lambda_1}{\lambda_0} d$.

For the first case, let $\lambda_0 x$ be the intersection point, then $d(\lambda_0 x, \partial \Gamma(\mathfrak{b}_{\bullet})) = 0$, similarly we can construct the contradiction.

\end{proof}

\begin{corollary}\label{thm: res num = inf NO}
Let $R$ be a Noetherian $\kk$-domain, and define a valuation $\nu$ with one-dimensional leaves on $R$, and $S(R) = \ZZ^n_{\geq 0}$. Let $\mathfrak{a}_{\bullet}$ and $\mathfrak{b}_{\bullet}$ be graded families of ideals of $R$. Suppose $\inf \{ d(x_1, x_2) ~|~ x_1 \in \Gamma(\mathfrak{b}_{\bullet}), x_2 \in \partial \mathbb{R}_{\geq 0}^n \} > 0$, then we have the following equality:
\begin{align*}
\hat{\rho}(in_{\nu}(\mathfrak{a}_{\bullet}), in_{\nu}(\mathfrak{b}_{\bullet})) &= \inf \{ \lambda > 0 ~|~ \lambda \cdot \Gamma(\mathfrak{a}_{\bullet}) \subseteq \Gamma(\mathfrak{b}_{\bullet}) \} \\
&= \sup \{ \lambda > 0 ~|~ \lambda \cdot \Gamma(\mathfrak{a}_{\bullet}) \not\subseteq \Gamma(\mathfrak{b}_{\bullet}) \}
\end{align*}

\end{corollary}

\begin{proof}
Combining \Cref{lemma: distance of two regions} and the proof of the second statement of \Cref{theorem: resurgence number = lambda}, we have $\hat{\rho}(in_{\nu}(\mathfrak{a}_{\bullet}), in_{\nu}(\mathfrak{b}_{\bullet})) \leq \inf \{ \lambda > 0 ~|~ \lambda \cdot \Gamma(\mathfrak{a}_{\bullet}) \subseteq \Gamma(\mathfrak{b}_{\bullet}) \}$, so asymptotic resurgence number $\hat{\rho}(in_{\nu}(\mathfrak{a}_{\bullet}), in_{\nu}(\mathfrak{b}_{\bullet}))$ always exists. Applying first statement of \Cref{theorem: resurgence number = lambda}, then we are done.

\end{proof}

\begin{remark}
Observe that for any pair of families of ideals $(\mathfrak{a}_{\bullet}, \mathfrak{b}_{\bullet})$, $\inf \{ \lambda > 0 ~|~ \lambda \cdot \Gamma(\mathfrak{a}_{\bullet}) \subseteq \Gamma(\mathfrak{b}_{\bullet}) \}$ (and $\sup \{ \lambda > 0 ~|~ \lambda \cdot \Gamma(\mathfrak{a}_{\bullet}) \not\subseteq \Gamma(\mathfrak{b}_{\bullet}) \}$) always exists, but $\hat{\rho}(in_{\nu}(\mathfrak{a}_{\bullet}), in_{\nu}(\mathfrak{b}_{\bullet}))$ may not exists. However, under assumptions of \Cref{thm: res num = inf NO}, $\hat{\rho}(in_{\nu}(\mathfrak{a}_{\bullet}), in_{\nu}(\mathfrak{b}_{\bullet}))$ always exists. 
\end{remark}

\begin{example}\label{eg no asy res num}
\Cref{thm: res num = inf NO} implies that the statement may fails if $\Gamma(\mathfrak{b}_{\bullet})$ approaches or intersects the boundary of $\mathbb{R}^n_{\geq 0}$. We can give a simple example when it fails. Given a polynomial over a field $R = \kk [x,y,z]$, define $\mathfrak{b}_n = (x^{\floor{\ln(n+2)}} y^n z^{2n})$ for the family of ideals $\mathfrak{b}_{\bullet}$, and define $\mathfrak{a}_n = (y^n z^{2n})$ for the family of ideals $\mathfrak{a}_{\bullet}$. It is not hard to see $\Gamma(\mathfrak{b}_{\bullet})=\Gamma(\mathfrak{a}_{\bullet})$, and $\Gamma(\mathfrak{b}_{\bullet})$ intersects the boundary of $\mathbb{R}^n_{\geq 0}$, but $\hat{\rho}(\mathfrak{a}_{\bullet}, \mathfrak{b}_{\bullet})$ does not exist.

\end{example}

Given a family of monomial ideals $\mathfrak{a}_{\bullet}$ in a polynomial ring over a field $R = \kk [x_1, ... \, , x_n]$, observe that $\inf \{ d(x_1, x_2) ~|~ x_1 \in \Gamma(\mathfrak{b}_{\bullet}), x_2 \in \partial \mathbb{R}_{\geq 0}^n \} > 0$ if only if there exists an $i \in \{1, ... \, , n\}$ such that $0 = inf \{\frac{u_i}{n} ~|~ (u_1, ... \, , u_n)$ is the degree of the monomial generator of $\mathfrak{a}_n \}$, so we have another version of \Cref{thm: res num = inf NO} that only states via algebraic conditions on families of monomial ideals but not via convex geometry:

\begin{corollary}\label{thm: res num = inf NO alg con}
Let $\mathfrak{a}_{\bullet}$ and $\mathfrak{b}_{\bullet}$ be graded families of monomial ideals polynomial ring over a field, and if there is no $i \in \{1, ... \, , n\}$ such that $0 = inf \{\frac{u_i}{n} ~|~ (u_1, ... \, , u_n)$ is the degree of the monomial generator of $\mathfrak{b}_n$, then we have the following equality:
\begin{align*}
\hat{\rho}(\mathfrak{a}_{\bullet}, \mathfrak{b}_{\bullet}) = \inf \{ \lambda > 0 ~|~ \lambda \cdot \Gamma(\mathfrak{a}_{\bullet}) \subseteq \Gamma(\mathfrak{b}_{\bullet}) \} = \sup \{ \lambda > 0 ~|~ \lambda \cdot \Gamma(\mathfrak{a}_{\bullet}) \not\subseteq \Gamma(\mathfrak{b}_{\bullet}) \}
\end{align*}

\end{corollary}

Other than the two results having assumptions on \emph{regions}, We can generalize \cite[Theorem 2.7]{rncb} from the assumption that  $R$ is a polynomial ring to $R$ is a Noetherian $\kk$-domain, and it is \Cref{res = region in Notherian case}. But before that, we need to state the following theorem that generalizes some parts of \cite[Theorem 3.1]{ntbr} and \cite[Theorem 3.4]{ntbr}, which is the main tool for \Cref{res = region in Notherian case}.

\begin{theorem}\label{main tool for Noetherian case}
Let $\mathfrak{a}_{\bullet}$ be a graded family of ideals in a Noetherian $\kk$-domain $R$, and define a valuation $\nu$ with one-dimensional leaves on $R$. Then the following statements are equivalent:
\begin{itemize}
\item[(1)] There exist an integer $c$ such that $\Gamma(\mathfrak{a}_{\bullet}) = \frac{1}{c} VP(\mathfrak{a}_c) $.
\item[(2)] There exist an integer $c$ such that $\frac{1}{c} VP(\mathfrak{a}_c) = \frac{1}{ck} VP(\mathfrak{a}_{ck})$ for all $k \geq 1$.
\item[(3)] There exist an integer $c$ such that $in_{\nu}(\mathfrak{a}_c)^k = in_{\nu}(\mathfrak{a}_{ck})$ for all $k \geq 1$.
\item[(4)] The Rees algebra associated to $in_{\nu}(\mathfrak{a}_{\bullet})$, denoted as $\R(in_{\nu}(\mathfrak{a}_{\bullet}))$, is Noetherian.
\end{itemize}
\end{theorem}

\begin{proof}
The proof follows the every step of proof of \cite[Theorem 3.1]{ntbr} and proof of \cite[Theorem 3.4]{ntbr}. The original theorem assume $R$ is a polynomial ring, we want to extended it to Noetherian $\kk$-domain. Noetherianity in the proof of \cite[Theorem 3.4]{ntbr} is equivalent to finite generation of the Rees algebra. \cite[Proposition 1.5.4]{cmr} states that graded $S_0$-algebra $S$ is Noetherian if only if $S_0$ is Noetherian and $S$ is finitely generated. Therefore, consider $S_0$ as our $R$, and $S$ as our $\R(in_{\nu}(\mathfrak{a}_{\bullet}))$, then it passes the Noetherianity to finite generation. Therefore we can generalize statements of \cite[Theorem 2.1]{HERZOG2007304} to Noetherian $\kk$-domain case, together with the proof of \cite[Theorem 3.1]{ntbr} and proof of \cite[Theorem 3.4]{ntbr}, to finish the proof.
\end{proof}

\begin{theorem}\label{res = region in Notherian case}
Let $R$ be a Noetherian $\kk$-domain, and define a valuation $\nu$ with one-dimensional leaves on $R$. Given a pair of graded families of ideals $(\mathfrak{a}_{\bullet}, \mathfrak{b}_{\bullet})$ of $R$, if the Rees algebra $\mathcal{R}(in_{\nu}(\mathfrak{b}_{\bullet}))$ is Noetherian, then
\begin{align*}
\hat{\rho}(in_{\nu}(\mathfrak{a}_{\bullet}), in_{\nu}(\mathfrak{b}_{\bullet})) &= \inf \{ \lambda > 0 ~|~ \lambda \cdot \Gamma(\mathfrak{a}_{\bullet}) \subseteq \Gamma(\mathfrak{b}_{\bullet}) \} \\
&= \sup \{ \lambda > 0 ~|~ \lambda \cdot \Gamma(\mathfrak{a}_{\bullet}) \not\subseteq \Gamma(\mathfrak{b}_{\bullet}) \}
\end{align*}

\end{theorem}

\begin{proof}
The proof follows every step of \cite[Theorem 2.7]{rncb} and \Cref{main tool for Noetherian case}.
\end{proof}

The generalization of \cite[Theorem 2.7]{rncb} from polynomial rings to Noetherian $\kk$-domains is considerably stronger than one might initially expect. Indeed, the Noetherianity of the Rees algebra associated with a family of ideals is typically studied in Noetherian rings instead of polynomial rings. Moreover, a number of important classes of families of ideals are known to admit Noetherian Rees algebras, as summarized below.

\begin{example}\label{stable filtration}
Given a \emph{filtration} $\mathfrak{a}_{\bullet} = \{ \mathfrak{a}_{n} \}_{n \in \ZZ_{\geq 0}}$ in a Noetherian ring, recall the definition of a \emph{$I$-stable filtration} (or \emph{$I$-good filtration}) if $\mathfrak{a}_{\bullet}$ satisfies the following:
\begin{itemize}
\item[(1)] $I \mathfrak{a}_{n} \subseteq \mathfrak{a}_{n+1}$ for all $n \geq 0$.
\item[(2)] $I \mathfrak{a}_{n} = \mathfrak{a}_{n+1}$ for all $n \gg 0$.
\end{itemize}
Let $\R(\mathfrak{a}_{\bullet})$ be the Rees algebra associated to $\alpha_{\bullet}$, and $\R(I)$ be the Rees algebra associated to the \emph{$I$-adic filtration} (i.e. \emph{filtration} of ordinary powers of $I$). By \cite[Remark 2.2]{cmg}, If $\alpha_{\bullet}$ is an $I$-stable filtration, then $\R(\mathfrak{a}_{\bullet})$ is finitely generated as $\R(I)$-module. It is a well-known fact that $\R(I)$ is Noetherian when $I$ is an ideal in a Noetherian ring, so $\R(\mathfrak{a}_{\bullet})$ is also Noetherian. Therefore Rees algebra associated to an $I$-stable filtration in a Notherian ring is Noetherian.

\end{example}

\begin{example}\label{integral closure filtration}
Let $R$ be analytically unramified Noetherian local ring, i.e. The \emph{completion} of $R$ with respect to the maximal ideal is \emph{reduced}. Let $I$ be an ideal of $R$, and let $\R(\bar{I})$ be the Rees algebra associated to the \emph{integral closure filtration of $I$} (i.e. \emph{filtration} of integral closure of ordinary powers of $I$), then by \cite[Corollary 9.2.1]{ICI}, $\R(\bar{I})$ is an \emph{$I$-stable filtration}, so by \Cref{stable filtration}, it is Noetherian.
\end{example}

	\bibliographystyle{alpha}
	\bibliography{References}

@article{9,
title = {On the resurgence and asymptotic resurgence of homogeneous ideals},
journal = {Mathematische Zeitschrift},
volume = {302},
pages = {2407-2434},
year = {2022},
author = {Jayanthan, A. V. and Kumar, Arvind and Mukundan, Vivek
},
}

@article{8,
title = {Asymptotic resurgences for ideals of positive dimensional subschemes of projective space},
journal = {Advances in Mathematics},
volume = {246},
pages = {114-127},
year = {2013},
author = {Elena Guardo and Brian Harbourne and Adam {Van Tuyl}},
}

@article{7,
author = {Grifo, Eloisa and Huneke, Craig and Mukundan, Vivek},
title = {Expected resurgences and symbolic powers of ideals},
journal = {Journal of the London Mathematical Society},
volume = {102},
number = {2},
pages = {453-469},
year = {2020}
}

@article{6,
title = {Resurgences for ideals of special point configurations in $\mathbb{P}^n$ coming from hyperplane arrangements},
journal = {Journal of Algebra},
volume = {443},
pages = {383-394},
year = {2015},
author = {M. Dumnicki and B. Harbourne and U. Nagel and A. Seceleanu and T. Szemberg and H. Tutaj-Gasinska},
}

@article{5,
title = {Duality for asymptotic invariants of graded families},
journal = {Advances in Mathematics},
volume = {430},
pages = {109208},
year = {2023},
author = {Michael DiPasquale and Thai Thanh Nguyen and Alexandra Seceleanu},
}

@article{4,
title = {On resurgence via asymptotic resurgence},
journal = {Journal of Algebra},
volume = {587},
pages = {64-84},
year = {2021},
author = {DiPasquale, Michael and Drabkin, Ben},
}

@article{3,
author = {Dipasquale, Michael and Francisco, Christopher and Mermin, Jeffrey and Schweig, Jay},
year = {2019},
pages = {6655-6676},
volume = {372},
title = {Asymptotic resurgence via integral closures},
journal = {Transactions of the American Mathematical Society},
}

@article{2,
title = {Demailly's Conjecture and the containment problem},
journal = {Journal of Pure and Applied Algebra},
volume = {226},
number = {4},
pages = {106863},
year = {2022},
author = {Sankhaneel Bisui and Eloisa Grifo and Huy Tai Ha and Thai Thanh Nguyen},
}

@article{1,
author = {Bisui, Sankhaneel and Grifo, Eloisa and Ha, Tai and Nguyen, Thai},
year = {2022},
month = {05},
pages = {371-394},
title = {Chudnovsky’s conjecture and the stable Harbourne–Huneke containment},
volume = {9},
journal = {Transactions of the American Mathematical Society, Series B},
}

@article{rncb,
  title={Resurgence numbers and convex regions associated to pairs of graded families of ideals},
  author={Tai Huy Ha and A.V. Jayanthan and Arvind Kumar and Thai Thanh Nguyen},
  journal = {Algebraic Combinatorics},
  pages = {649--663},
  year = {2026},
  publisher = {The Combinatorics Consortium},
  volume = {9},
  number = {3},
}

@article{HA2026,
title = {Resurgence number of graded families of ideals},
journal = {Journal of Algebra},
year = {2026},
doi = {https://doi.org/10.1016/j.jalgebra.2026.04.015},
author = {Tai Huy Ha and Arvind Kumar and Hop D. Nguyen and Thai Thanh Nguyen}
}

@article{ntbr,
author = {Ha, Tai and Nguyen, Thai},
year = {2024},
month = {08},
pages = {1065-1097},
title = {Newton–Okounkov body, Rees algebra, and analytic spread of graded families of monomial ideals},
volume = {11},
journal = {Transactions of the American Mathematical Society, Series B},
doi = {10.1090/btran/177}
}

@article{AttouchWets,
author = {Beer, Gerald},
year = {2008},
month = {09},
title = {The Attouch-Wets topology in metric and normed spaces},
volume = {4},
journal = {Pacific Journal of Optimization},
}

@book{CommutativeAlgebra,
 ISBN = {9780387942681},
 author = {Eisenbud, David},
 publisher = {Springer New York, NY},
 title = {Commutative Algebra with a View Toward Algebraic Geometry},
 year = {1995}
}

@book{Convexity,
 ISBN = {9781108837590},
 author = {Basu, Amitabh},
 publisher = {Cambridge University Press},
 title = {Convexity and its Applications in Discrete and Continuous Optimization},
 year = {2025}
}

@book{cmr,
 ISBN = {9780511608681},
 author = {Bruns, Winfried and Herzog, H. Jürgen},
 publisher = {Cambridge University Press},
 title = {Cohen-Macaulay Rings},
 year = {1998}
}

@article{HERZOG2007304,
title = {Symbolic powers of monomial ideals and vertex cover algebras},
journal = {Advances in Mathematics},
volume = {210},
number = {1},
pages = {304-322},
year = {2007},
issn = {0001-8708},
author = {Jürgen Herzog and Takayuki Hibi and Ngô Viêt Trung},
}

@article{khovanskii,
title = {Khovanskii Bases, Higher Rank Valuations, and Tropical Geometry},
author = {Kaveh, Kiumars and Manon, Christopher},
journal = {SIAM Journal on Applied Algebra and Geometry},
volume = {3},
number = {2},
pages = {292-336},
year = {2019},
}

@book{gbca,
 ISBN = {9780821872871},
 author = {Ene, Viviana and Herzog, Jurgen},
 publisher = {American Mathematical Society},
 title = {Grobner Bases in Commutative Algebra},
 year = {2011}
}

@article {EinLazarsfeldSmith,
	AUTHOR = {Ein, Lawrence and Lazarsfeld, Robert and Smith, Karen E.},
	TITLE = {Uniform bounds and symbolic powers on smooth varieties},
	JOURNAL = {Invent. Math.},
	FJOURNAL = {Inventiones Mathematicae},
	VOLUME = {144},
	YEAR = {2001},
	NUMBER = {2},
	PAGES = {241--252},
	ISSN = {0020-9910,1432-1297},
}

@article {HochsterHuneke,
	AUTHOR = {Hochster, Melvin and Huneke, Craig},
	TITLE = {Comparison of symbolic and ordinary powers of ideals},
	JOURNAL = {Invent. Math.},
	FJOURNAL = {Inventiones Mathematicae},
	VOLUME = {147},
	YEAR = {2002},
	NUMBER = {2},
	PAGES = {349--369},
	ISSN = {0020-9910,1432-1297},
}

@article{BocciHarbourne2010,
  title     = {Comparing powers and symbolic powers of ideal},
  author    = {Cristiano Bocci and Brian Harbourne}, 
  journal   = {Journal of Algebraic Geometry},
  volume    = {19},
  number    = {3},
  pages     = {399--417},
  year      = {2010}
}

@article{VILLARREAL2023103656,
title = {A duality theorem for the ic-resurgence of edge ideals},
journal = {European Journal of Combinatorics},
volume = {109},
pages = {103656},
year = {2023},
issn = {0195-6698},
author = {Rafael H. Villarreal},
}

@article{nob,
title = {Newton-Okounkov bodies, semigroups of integral points, graded algebras and intersection theory},
journal = {Annals of Mathematics},
volume = {176},
pages = {925-978},
year = {2012},
issn = {1939-8980},
author = {Kaveh, Kiumars and Khovanskii, A.G.},
}

@article{bodiesmultiplicity,
title = {Convex bodies and multiplicities of ideals},
journal = {Proceedings of the Steklov Institute of Mathematics},
volume = {286},
pages = {268–284},
year = {2014},
issn = {1531-8605},
author = {Kaveh, Kiumars and Khovanskii, A.G.},
}

@article{MultiplicitiesAT,
  title={Multiplicities associated to graded families of ideals},
  author={Steven Dale Cutkosky},
  journal={Algebra \& Number Theory},
  year={2012},
  volume={7},
  pages={2059-2083},
  issn={1944-7833},
}

@article{CUTKOSKY201455,
title = {Asymptotic multiplicities of graded families of ideals and linear series},
author = {Steven Dale Cutkosky},
journal = {Advances in Mathematics},
volume = {264},
pages = {55-113},
year = {2014},
issn = {0001-8708},
}

@book{ICI,
 ISBN = {9780521688604},
 author = {Swanson,Irena and Huneke,Craig},
 publisher = {Cambridge University Press},
 title = {Integral Closure of Ideals, Rings, and Modules},
 year = {2006}
}

@article{cmg,
author = {Heinzer, William and Kim, Mee-Kyoung and Ulrich, Bernd},
year = {2011},
month = {10},
pages = {3547-3580},
title = {The Cohen–Macaulay and Gorenstein Properties of Rings Associated to Filtrations},
volume = {10},
journal = {Communications in Algebra},
}

\end{document}